\documentclass[a4paper,11pt]{article} 

\title{Energy of the Coulomb gas on the sphere at low temperature} 

\author{
\ Carlos Beltr\'an\footnote{
Departamento MATESCO, Universidad de Cantabria. Avda. Los Castros s/n, Santander, Spain. Email: \href{mailto:beltranc@unican.es}{\nolinkurl{beltranc@unican.es}}. Partially supported by MICINN grants MTM2017-83816-P and MTM2017-90682-REDT and by Banco de Santander-Universidad de Cantabria grant 21.SI01.64658. 
}\,,
\qquad 
Adrien Hardy\footnote{
Laboratoire Paul Painlev\'e, Universit\'e de Lille,
Cit\'e Scientifique, 59655 Villeneuve d'Ascq Cedex, 
France. 
Email: \href{mailto:adrien.hardy@math.univ-lille1.fr}{\nolinkurl{adrien.hardy@univ-lille.fr}}. Partially supported by  ANR JCJC {BoB} (ANR-16-CE23-0003) and Labex {CEMPI} (ANR-11-LABX-0007-01).
}}

\usepackage{natbib}
\usepackage{anysize}
\marginsize{3cm}{3cm}{2.5cm}{2.5cm}
\usepackage[english]{babel}
\usepackage{fancyhdr}
\usepackage{amssymb,amsmath,amscd,amsfonts,amsthm,bbm}
\usepackage{amsfonts}
\usepackage{amsmath}
\usepackage{mathtools}
\usepackage{mathrsfs} 
\usepackage{latexsym,lmodern,mathtools,color}
\usepackage{tikz,graphicx,subfigure,overpic}
\usepackage[breaklinks,colorlinks=true, linkcolor=blue, citecolor=blue]{hyperref}
\usepackage[fixlanguage]{babelbib}
\usepackage{color}
\usepackage{comment}
\usepackage{fancyhdr} 

\usepackage{url}
\DeclareUrlCommand\email{\urlstyle{rm}}

\numberwithin{equation}{section}

\newtheorem{theorem}{Theorem}[section]
\newtheorem{lemma}[theorem]{Lemma}
\newtheorem{corollary}[theorem]{Corollary}
\newtheorem{proposition}[theorem]{Proposition}

\theoremstyle{definition} 

\newtheorem{Remark}[theorem]{Remark}

\newcommand{\eq}{\begin{equation}}
\newcommand{\qe}{\end{equation}}

\newcommand{\R}{\mathbb{R}}
\renewcommand{\S}{\mathbb{S}}
\newcommand{\C}{\mathbb{C}}
\newcommand{\E}{\mathbb{E}}
\renewcommand{\P}{\mathbb{P}}
\newcommand{\cO}{\mathcal{O}}
\renewcommand{\d}{ {\mathrm d}}
\newcommand{\e}{\mathrm{e}}

\newcommand{\HN}{\mathscr{H}_N}
\renewcommand{\H}{\mathscr{H}}

\renewcommand{\epsilon}{ \varepsilon}
\renewcommand{\phi}{ \varphi}

\begin{document}
\maketitle

\begin{abstract} We consider the Coulomb gas of $N$ particles on the sphere and show that the logarithmic energy of the configurations approaches the minimal energy up to an error of order $\log N$, with exponentially high probability and on average, provided the temperature is  $\cO(1/N)$.  
\end{abstract}

\section{Introduction and statement of the result}
\label{sec:intro}

\paragraph{Smale's 7th problem.}
Let $\|\cdot\|$ be the Euclidean norm on $\R^3$ and $$\S:=\big\{x\in\R^3: \; \|x\|= 1\big\}$$ the unit sphere. Consider the logarithmic energy of a configuration $x_1,\ldots,x_N\in\S$,
$$
\HN(x_1,\ldots,x_N):=\sum_{i\neq j}\log\frac1{\|x_i-x_j\|}.
$$
The 7th problem from \cite{Sma00}'s list of mathematical problems for the next century asks to find for every $N\geq 2$ a configuration $x_1,\ldots,x_N\in\S$ and  $c>0$ independent on $N$ such that
\eq
\label{Smale}
\HN(x_1,\ldots,x_N)-\min_{\S^N}\HN \leq  c\log N.
\qe
More precisely, quoting Smale:  ``For a precise version one could ask for a real number algorithm in the sense of \cite*{BCSS}
which on input $N$ produces as output distinct points $x_1,\ldots,x_N$ on the 2-sphere satisfying \eqref{Smale} with
halting time polynomial in $N$''.

\paragraph{Large $N$ expansion.}One difficulty in this problem is that the large $N$ behavior of $\min_{\S^N}\HN$ is not even known up to precision $\log N$. Indeed, the actual knowledge is that 
\begin{equation}\label{eq:as}
\min_{\S^N}\HN=V_{\log}(\S)\,N^2-\frac12\,N\log N+C_{\log}\,N+o(N),\qquad N\to\infty,
\end{equation}
where the constant $$V_{\log}(\S)=\min_{\mu\in\mathcal P(\S)}\iint \log\frac{1}{\|x-y\|}\,\mu(\d x)\mu(\d y)=\frac12-\log2$$ is the minimal logarithmic energy over the space $\mathcal P(\S)$ of probability measures on $\S$.

The exact value of the constant $C_{\log}$ in \eqref{eq:as} is still conjectural. A series of papers by \cite*{Wagner,RSZ94,Dubickas} and \cite{Brauchart2008} gave upper and lower bounds for $C_{\log}$ as well as similar bounds for other choices of energies, but indeed the existence of $C_{\log}$ has  only recently been obtained by \cite{BS18} where it is expressed in terms of the minimum of the renormalized energy introduced by \cite{SS12}. In the same paper it is also proved that
\begin{equation}\label{eq:upper}
C_{\log}\leq 2\log 2 +\frac12\log\frac23+3\log\frac{\sqrt\pi}{\Gamma(1/3)}=-0.0556053\ldots.
\end{equation}
This upper bound is conjectured to be an equality, see \cite*{BHS12}. \cite{BS18} have also shown this conjecture is equivalent to the conjecture that the triangular lattice minimizes the renormalized energy. The tightest known lower bound $C_{\log}\geq -0.2232823526\ldots$, proved by \cite{Dubickas}, seems thus to be far from optimal.\\

One may look for configurations of $N$ points defined deterministically on $\S$ which attain small values for $\HN$, but it turns out to be very difficult to compute explicit asymptotics for any reasonable choice of points; \cite*{SaffSurvey} made many numerical experiments analyzing different constructions, but none of them seems to reach the upper bound for $C_{\log}$ in \eqref{eq:upper}.

\paragraph{Random configurations on the sphere.}

A possible strategy to  attack the  problem is to look for random configurations on $\S$ whose logarithmic energy could satisfy \eqref{Smale} on average, or with high probability. If one naively picks $x_1,\ldots,x_N$  at random uniformly and independently on $\S$, then an easy computation yields a formula for the mean energy,
$$
\E_{\text{Unif}}\big[\HN\big(x_1,\ldots,x_N)]=V_{\log}(\S) \, N^2+(\log2-1/2)N.
$$ 
Thus,  unstructured configurations do not even reach precision $N$ on average.  \cite*{ABS11} suggested instead to take for $x_i$'s the zeros of Elliptic  polynomials after  a stereographic projection. These are random polynomials on $\C$  defined by
$$
P_N(z):=\sum_{k=0}^{N}\sqrt{{N}\choose{k}}\,\xi_k \,z^k
$$
where the $\xi_k$'s are i.i.d standard complex gaussian random variables. Up to multiplication by non-vanishing holomorphic functions, they are the only gaussian analytic functions (GAF) whose zeros are invariant under the isometries of the sphere; they are also known as the spherical GAFs, see  \citep{HKPV09}.  \cite{ABS11} proved that the mean energy of these random configurations equals
$$
\E_{\text{GAF}}\big[\HN(x_1,\ldots,x_N)\big]=V_{\log}(\S)\,N^2-\frac12\,N\log N+(\log2 -1/2  )N
$$
 which reaches the precision $N$, but not more; note that $\log2-1/2=0.1931472\ldots$  See also   \citep{Zhong,ZhongZelditch, ZeZe10, Butez, BuZe17} for related results on the zeros of random polynomials.
 
 Another natural attempt at using random configurations for this problem is to consider the determinantal point process (DPP) on $\C$ known as the spherical ensemble in random matrix theory. The simplest description of the spherical ensemble, due to \cite{Krishnapur}, is as follows: choose two matrices $A,B$ whose entries are i.i.d standard complex gaussian random variables and compute the $N$ eigenvalues of $A^{-1}B$. Up to a stereographic projection, it turns out that these points are quite well distributed on $\S$ on average. Indeed, \cite{AZ15} obtained for these random configurations that, as $N\to\infty$,
$$
\E_{\text{DPP}}\big[\H_N(x_1,\ldots,x_N)\big]=V_{\log}(\S)\,N^2-\frac12\,N\log N+ \big(\log 2-\gamma/2\big)\,N-\frac14+\cO(1/N),
$$
where $\gamma$ is the Euler constant; hence $\log 2-\gamma/2=0.4045393\ldots$ Finally, a particular random construction with seemingly small energy values based on the distribution of charges along parallels in $\S$ is currently being studied by \cite{Ujue}.

All these bounds (analytical and numerical) are still far from the upper bound in \eqref{eq:upper}.

\paragraph{The Coulomb gas on the sphere.}  Another natural random configuration associated with this problem is the Coulomb gas on $\S$, which is the main character of this work; for references see e.g. \cite[Section 15.6]{For10}. More precisely, let $\sigma$ be the uniform measure on $\S$ normalized so that $\sigma(\S)=1$, namely $\sigma:=(4\pi)^{-1}\mathrm{Vol}$.  For any $N\geq 2$ and $\beta >0$,  consider the probability measure on $\S^N$,
$$
\P_{N,\beta}(\d x):=\frac1{Z_{N,\beta}}\,\e^{-\beta \HN(x)} \sigma^{\otimes N}(\d x),
$$ 
where we introduced the normalisation constant known as the partition function,
$$
{Z_{N,\beta}}:=\int_{\S^N}\e^{-\beta \HN(x)}\sigma^{\otimes N}(\d x).
$$
We denote by $\E_\beta$ the expectation with respect to $\P_{N,\beta}$. Here $\HN(x)$ means $\HN(x_1,\ldots,x_N)$ when $x=(x_1,\ldots,x_N)\in\S^N.$  Physically, $\HN(x)$ represents the electrostatic energy of a configuration $x_1,\ldots,x_N$ of $N$ identical charges placed on the sphere, following the classical laws of 2D electrostatics.  $\P_{N,\beta}$ is known in statistical physics as the canonical Gibbs measure associated with this energy and the random configurations it generates are referred to as the Coulomb gas at inverse temperature $\beta$. Typical configurations of the Coulomb gas will try to minimize $\HN$ because of its density distribution proportional to $\e^{-\beta \HN}$. It is thus tempting to evaluate the energy $\HN(x)$ for such random configurations so as to approximate the minimum  of $\HN$. In fact, when $\beta=1$ the Coulomb gas benefits from an integrable structure: up to stereographic projection, this is the spherical ensemble mentioned above and studied by \cite{AZ15}.  But the larger $\beta$ is the more likely it is for $\P_{N,\beta}$ to generate a configuration close to a minimizer, although the determinantal structure is lost when $\beta\neq1$ making exact computations out of reach. The main achievement of this work is to show that the Coulomb gas on the sphere at temperature $\cO(1/N)$ provides almost minimizing configurations in the sense of Smale's problem with high probability as well as on average.

\begin{theorem} 
\label{th:main}
For any $N\geq 2$ and any $\beta\geq 1$, let $x_1,\ldots,x_N$ be the random configuration on $\S$ with joint distribution $\P_{N,\beta}$. For any constant $c>0$ we have
\eq
\HN(x_1,\ldots,x_N)-\min_{\S^N}\HN \leq c\log N
\qe
with probability at least 
$1-\e^{-\kappa N}$,
where
$$ \kappa:=c\,\frac\beta N\log N -\log\beta -8\log N.
$$
Moreover, the mean energy satisfies
$$
\E_{\beta}\big[\HN(x_1,\ldots,x_N)\big]-\min_{\S^N}\HN\leq \frac N\beta \Big(\log\beta +8\log N\Big).
$$
\end{theorem}

Note that given $\beta$ and $N$, the constant $c$ has to be chosen so that $\kappa>0$ since otherwise the first result becomes trivial. We reach the precision $\log N$  for any $N\geq 2$ when $\beta$ is at least of order $N$. For example, by taking $\beta=N$ and $c=10$ in Theorem \ref{th:main}, we obtain the following estimates.

\begin{corollary}  For any $N\geq 2$, if the random configuration $x_1,\ldots,x_N$ on $\S$ has for distribution the Coulomb gas $\P_{N,N}$ at inverse temperature $\beta=N$, then 
$$
\HN(x_1,\ldots,x_N)-\min_{\S^N}\HN \leq 10\log N
$$
with probability at least $1-\e^{-N\log N}$. Moreover,
$$
\E_{N}\big[\HN(x_1,\ldots,x_N)\big]-\min_{\S^N}\HN\leq 9\log N.
$$
\end{corollary}

Thus, if one accepts stochastic algorithms  as solutions for the precise version of Smale's 7th problem, it remains to show that one can sample a configuration from $\P_{N,N}$ in polynomial time, or at least approximate configurations which are close from those of $\P_{N,N}$, say, in total variation, with high probability.  Of course, letting $\beta$ growing with $N$ faster than a linear rate leads to improved convergence results, and even allows $c$ to decay to zero as $N\to\infty$, but we expect that the larger $\beta$ is the  harder it is to sample such configurations in practice.

The proof of the theorem relies on two facts. First, a general concentration inequality for the energy of arbitrary Gibbs measures provided in Section \ref{sec:Gibbs}: we observe in a general setting that an explicit control on the probability that $\HN(x)-\min\HN>\delta$, as well as an upper bound on the mean energy $\E_{\beta}[\HN]$, can be made using solely a lower bound  for $\log Z_{N,\beta}+\beta\min\HN$. This lower bound can be easily derived when an upper bound on the second derivative of the energy is known.  In Section \ref{sec:Spherebound}, we work out such an upper bound in the case of the Coulomb gas on $\S$ and prove Theorem \ref{th:main}.

\section{Concentration for the Gibbs measure's energy}
\label{sec:Gibbs}

In this section, we consider the following general setting: Let $S$ be any non-empty measurable space equipped with a probability measure $\mu$ and a measurable map $\H:S\to\R\cup\{+\infty\}$ such that $\inf_S\H>-\infty$. Consider for any $\beta>0$ the probability measure,
\eq
\label{Pbeta}
\P_\beta(\d x):=\frac{1}{Z_\beta}\,\e^{-\beta\H(x)}\mu(\d x),\qquad Z_\beta:=\int \e^{-\beta\H(x)}\mu(\d x),
\qe
and assume that $Z_\beta$ is finite and does not vanish so that $\P_\beta$ is well defined. In particular $\inf_S\H<+\infty.$ In the following, $\E_\beta$ stands for the expectation with respect to $\P_\beta$.

We first observe that one can relate the deviations of the random variable $\H(x)$ from  $\inf_S\H$, when $x$ has distribution $\P_\beta$,  to a lower bound on $\log Z_{\beta}  +\beta\inf_{S}\H$.

\begin{lemma}
\label{pr:Gibbs} Let  $C_\beta$ be any constant satisfying
\eq
\label{subtineq}
\log Z_{\beta}\geq  -\beta\inf_{S}\H- C_{\beta}.
\qe
Then, for any $\delta>0$,
\eq
\label{keyineq}
\P_{\beta}\Big(\H(x)-  \inf_{S}\H> \delta\Big)\leq \e^{-\beta \delta+ C_{\beta}}.
\qe
\end{lemma}

Note that since a rough upper bound yields $\log Z_\beta\leq  -\beta\inf_S\H$, the constant $C_\beta$ has to be non-negative, and $C_\beta=0$ if and only if $\H$ is $\mu$-a.s.  constant on $S$.
\begin{proof}
Indeed, since $\mu$ is a probability measure,
$$
\P_{\beta}\Big(\H(x)-  \inf_{S}\H> \delta\Big) \leq \frac1{Z_\beta}\,\e^{-\beta(\delta+\inf_S\H)} \leq  \,\e^{-\beta\delta+C_\beta},
$$
where we used \eqref{subtineq} for the second inequality. 
\end{proof}

The identity
 $$\E(X)=\int_0^\infty  \P(X>t)\d t\ ,$$ 
 which holds  for any positive random variable $X$, yields together with \eqref{keyineq} that
\eq
\label{Eest}
 \E_\beta\big[\H(x)\big]-\inf_S\H\leq \frac{\e^{ C_{\beta}}}{\beta}\ ,
\qe
and in particular  $\H(x)\in L^1(\P_\beta).$ Having in mind that $\beta$ may be taken large and that $C_\beta$ could grow with $\beta$,  one can obtain a better bound than \eqref{Eest} as follows.

\begin{lemma} 
\label{pr:Expbound}
Under the assumptions of Lemma \ref{pr:Gibbs}, 
\eq
\label{expectation} 
 \E_\beta\big[\H(x)\big]-\inf_{S}\H \leq \frac{C_{\beta}}\beta. 
\qe
\end{lemma}

\begin{proof} Let $0<\gamma<\beta$. Since $\H(x)\in L^1(\P_\beta)$ Jensen's inequality yields, 
$$
\log Z_\gamma=\log\int \e^{-(\gamma-\beta)\H(x)}\P_{\beta}(\d x)+\log Z_\beta\geq -(\gamma-\beta)\E_\beta\big[\H(x)\big]+\log Z_\beta
$$
and thus
$$
\E_\beta\big[\H(x)\big]\leq  \frac{\log Z_\gamma-\log Z_\beta}{\beta-\gamma}.
$$
Together with the rough upper bound
$$
\log Z_{\gamma}\leq -\gamma \inf_{S}\H
$$
and the definition of $C_\beta$, see  \eqref{subtineq}, we  obtain
 $$
 \E_\beta\big[\H(x)\big]\leq  \inf_{S}\H +\frac{ C_{\beta}}{\beta-\gamma}\, .
 $$
The lemma  follows by letting $\gamma\to 0$. 
\end{proof}

As we shall see in the next section, one way to obtain such a constant $C_\beta$ is to use an upper bound on the order two Taylor expansion of $\H$ near a minimizer.

\section{Proof of Theorem \ref{th:main}}
\label{sec:Spherebound}

Let $d_\S$ be the usual geodesic distance on the sphere $\S$,
$$
d_\S(x,y)=\arccos\langle x,y\rangle_{\R^3} \,,\qquad x,y\in\S\,,
$$
where $\langle \cdot,\cdot\rangle_{\R^3}$ stands for the usual inner product of $\R^3$. To prove Theorem \ref{th:main} we will use the following estimate as a key ingredient.

\begin{proposition}
\label{th:bh}
 Let $N\geq2$ and $x^*\in\S^N$ be a minimizer of $\HN$. If $x\in\S^N$ satisfies
 \[
\max_{1\leq i\leq N} d_\S(x_i,x_i^*)\leq\arcsin\left(\frac{s}{\sqrt 5 N^{3/2}}\right)
 \]
for some $0\leq s\leq \sqrt{5N}/2$, then  
 \[
  \HN(x)\leq \min_{\S^N}\HN+s^2.
 \]
\end{proposition}

A similar estimate is provided in \cite[Theorem 1.8]{Bel13}. Proposition~\ref{th:bh}  improves the range of validity of this result with an alternative proof.

We will also rely on the following result of \cite{Dra02} for the separation distance. 

\begin{proposition} 
\label{le:sep}
 Let $N\geq2$ and $x^*\in\S^N$ be a minimizer of $\HN$.  We have
$$
\min_{i\neq j}\|x_i^*-x_j^*\|\geq \frac2{\sqrt{N-1}}.
$$
\end{proposition}

 We are now in position to provide a proof for our main theorem.

\begin{proof}[Proof of Theorem \ref{th:main}] 

For any $0\leq r\leq\pi$ and $x\in\S$, the volume of a spherical cap $B_\S(x,r):=\{y\in\S:\, d_\S(x,y)\leq r\}$  is explicit: recalling $\sigma$ is  normalized so that $\sigma(\S)=1$, 
$$
\sigma\big(B_\S(x,r)\big)=\sin^2(\frac r2).
$$
In particular, for any $t\in(0,1)$,
$$
\sigma\Big(B_\S(x,\arcsin(t))\Big)=\sin^2\left(\frac12\arcsin(t)\right)=\frac{1-\sqrt{1-t^2}}2\geq  \frac{t^2}4.
$$
For any $N\geq 2$ and any minimizer $x^*\in\S^N$ of $\HN$ we set, for any $0<s\leq \sqrt{5N}/2$, 
$$
\Omega:=\left\{x\in\S^N:\, \max_{1\leq i\leq N} d_\S(x_i,x_i^*)\leq  \arcsin\left(\frac{s}{\sqrt 5 N^{3/2}}\right)\right\}.
$$
Thus,
$$
\sigma^{\otimes N}(\Omega)\geq \left( \frac{s^2}{20N^{3}}\right)^N.
$$
Next, assume $\beta\geq 1$ and we use Proposition \ref{th:bh} to obain the lower bound
\begin{align*}
 \log Z_{N,\beta}& \geq\log\int_{\Omega}\e^{-\beta\HN(x)}\,\sigma^{\otimes N }(\d x)\\
& \geq -\beta\min_{\S^N}\HN-\beta  s^2+\log\sigma^{\otimes N}(\Omega)\\
& \geq -\beta\min_{\S^N}\HN-\beta  s^2 +N\log\left( \frac{s^2}{20N^{3}}\right).
\end{align*}
Since this holds for any $0<s\leq \sqrt{5N}/2$ and $N\geq 2$, we obtain
\begin{align*}
 \log Z_{N,\beta}+\beta\min_{\S^N}\HN & \geq \max_{0<s\leq \sqrt{5N}/2}\left(-\beta  s^2 +N\log\left( \frac{s^2}{20N^{3}}\right)\right)\\
 &\geq  -N\big( 1+\log\beta+2\log N+\log20\big)\\
& \geq  -N\big( \log\beta+8\log N\big).
\end{align*}
We used that the maximum is reached at $s=\sqrt{N/\beta}$ and  $\sqrt{N/\beta}\leq \sqrt{5N}/2$ because $\beta\geq 1$. Thus, we have obtained the lower bound \eqref{subtineq} with 
$$
C_\beta=N\big(\log\beta+8\log N\big),$$ 
and Theorem  \ref{th:main} follows from Lemma \ref{pr:Gibbs} by taking $\delta=c\log N$ and Lemma \ref{pr:Expbound}.
\end{proof}

We finally turn to the proof of Proposition~\ref{th:bh}.

\begin{proof}[Proof of Proposition~\ref{th:bh}]From now, let $N\geq 2$, let $x^*\in\S^N$ be any minimizer of $\HN$ and  $0<t\leq 1/{(2N)}$. We set for convenience, for any $1\leq j\leq N$,
 \eq
 \label{Omega}
B_j:=B_\S(x_j^*,\arcsin(t))
 \qe
where we recall that $B_\S(x,r)$ is the ball of radius $r$ centered at $x\in\S$ associated with the geodesic distance $d_\S$. Since $\|x-y\|\leq d_\S(x,y)$ for any $x,y\in\S$ and $\arcsin(t)\leq \sqrt{2 t}$ for any $0<t\leq1/4$, Proposition \ref{le:sep} and the constraint on $t$ yields that the $\overline B_j$'s are disjoint subsets of $\S$ for any $N\geq 2$.

We equip the sphere $\S\subset\R^3$ with its usual Riemannian structure inherited from $\R^3$, whose associated distance is $d_\S$, and let $\Delta_\S$  be the associated Laplace-Beltrami operator. It is well known that $\log\|\cdot\|^{-1}$ satisfies the Poisson equation, namely
\eq
\label{Poisson}
\Delta_\S\log\frac1{\|\cdot\|}=2\pi(\sigma-\delta_0)
\qe 
in distribution, where we recall that $\sigma$ is the uniform probability measure on $\S$, see e.g. \cite[Section 15.6.1]{For10}.  It follows that, for any $p\in\S$, the map $F_p(x):=\log\|x-p\|^{-1}$ satisfies $\Delta_\S F_p(x)=1/2$ when $x\in\S\setminus\{p\}$, and in particular it is subharmonic there.   Thus, for any $(x_1,\ldots,x_N)\in B_1\times \cdots\times B_N$ and any $1\leq k\leq N$, the mapping
\eq
\label{composubh}
y\mapsto \H_N(x_1,\ldots,x_{k-1},y,x_{k+1},\ldots,x_N)
\qe
is  subharmonic on $B_k$ and satisfies the maximum principle. More precisely, by applying the classical Hopf's maximum principle for uniformly elliptic operators,  see e.g. \cite[Theorem 24.1]{Jost},   in the specific case of the Laplace-Beltrami operator $\Delta_\S$ (in coordinates), we have that for any open set $\Omega$ contained in a hemisphere of $\S$ and any $F\in C^2(\Omega)\cap C^0(\overline\Omega)$ which is subharmonic on $\Omega$, 
\[
\sup_{\Omega}F\leq \max_{\partial \Omega}F.
\]
By using $N$ times this inequality for the mappings \eqref{composubh}, we obtain
\eq
\label{subineq}
\max_{ B_1\times \cdots\times B_N}\HN \leq \max_{\partial B_1\times B_2\times \cdots\times B_N} \HN\leq \cdots \leq \max_{\partial B_1\times \partial B_2\times \cdots\times \partial B_N}\HN.
\qe
Next, we observe that
\begin{align}
\label{boundary}
& \partial B_1\times \partial B_2\times \cdots\times \partial B_N\nonumber\\
&  =\Big\{x\in\S^N:\; d_\S(x_i,x_i^*)=\arcsin(t) \text{ for all } 1\leq i\leq N\Big\}\nonumber\\
& =\Big\{\sqrt{1-t^2}x^*+tv:\; v\in \S^N,\; \langle x_i^*,v_i\rangle_{\R^3}=0  \text{ for all } 1\leq i\leq N\Big\}.
\end{align}
Indeed, a geodesic of $\S$ can always be parametrized as $\theta(u)=\sqrt{1-u^2}\, x+ u v$   for $x\in \S$ and $v\in\R^3$ satisfying $\|v\|=1$ and $\langle x,v\rangle_{\R^3}=0$; if $\Theta\subset\S$ is the curve $\{\theta(u)\}_{u\in[0,t]}$ then it starts at $\theta(0)=x$ with initial speed $\dot\theta(0)=v$ and has length 
$$
\text{Length}(\Theta)=\int_0^t\|\dot\theta(u)\|\,\d u=\int_0^t\frac{\d u}{\sqrt{1-u^2}}=\arcsin(t).
$$
In view of \eqref{Omega} and \eqref{subineq}--\eqref{boundary}, it is thus enough to show that, for any $v\in \S^N$ satisfying  $\langle x_i^*,v_i\rangle_{\R^3}=0$ for every $1\leq i\leq N$,
\eq
\label{toprove}
\HN(\sqrt{1-t^2}x^*+tv)\leq \HN(x^*)+5t^2N^3.
\qe
Indeed the proposition follows by setting $s:=\sqrt 5 tN^{3/2}$ which satisfies $0<s\leq \sqrt{5N}/2$.

To do so, let $v\in\S^N$ satisfying  $\langle x_i^*,v_i\rangle_{\R^3}=0$ for all $1\leq i\leq N$ and set $$g(t):=\HN\big(\sqrt{1-t^2}x^*+tv\big).$$ Since $g$ reaches a minimum at $t=0$,  there exists $\alpha\in(0,t)$ such that
\eq
\label{Taylor}
\HN\big(\sqrt{1-t^2}x^*+tv\big)=\HN(x^*)+\frac{t^2}2 \ddot g(\alpha).
\qe
Next, we set for convenience
$$
\gamma_{ij}(t):=\sqrt{1-t^2}(x_i^*-x_j^*)+t(v_i-v_j)
$$
so that we have
$$
g=-\sum_{i\neq j}\log \|\gamma_{ij}\|,\quad \dot g=-\sum_{i\neq j}\frac{\langle \gamma_{ij},\dot\gamma_{ij}\rangle}{ \|\gamma_{ij}\|^2}
$$
and moreover,  using the Cauchy-Schwarz inequality,
\eq
\label{boundgsec}
\ddot g  =\sum_{i\neq j}\left(
2\frac{\langle \gamma_{ij},\dot\gamma_{ij}\rangle^2}{ \|\gamma_{ij}\|^4}
-
\frac{\langle \gamma_{ij},\ddot\gamma_{ij}\rangle+\|\dot\gamma_{ij}\|^2}{ \|\gamma_{ij}\|^2}\right)\\
 \leq \sum_{i\neq j}\frac{\|\dot\gamma_{ij}\|^2-\langle \gamma_{ij},\ddot\gamma_{ij}\rangle}{\|\gamma_{ij}\|^2}\, .
\qe
By computing explicitly the derivatives of $\gamma_{ij}(t)$ we obtain
\begin{align*}
\|\gamma_{ij}(t)\|^2& = (1-t^2) \|x_i^*-x_j^*\|^2+2t\sqrt{1-t^2}\langle x_i^*-x_j^*,v_i-v_j\rangle+t^2\|v_i-v_j\|^2 \\
\|\dot\gamma_{ij}(t)\|^2& = \frac{t^2}{1-t^2}\|x_i^*-x_j^*\|^2-\frac{2t}{\sqrt{1-t^2}}\langle x_i^*-x_j^*,v_i-v_j\rangle+\|v_i-v_j\|^2\\
\langle \gamma_{ij}(t),\ddot\gamma_{ij}(t)\rangle& =  -\frac{1}{1-t^2}\|x_i^*-x_j^*\|^2-\frac{t}{(1-t^2)^{3/2}}\langle x_i^*-x_j^*,v_i-v_j\rangle
\end{align*}
and this yields with \eqref{boundgsec},
$$
\ddot g(t)\leq \sum_{i\neq j}\frac{
\frac{1+t^2}{1-t^2}\|x_i^*-x_j^*\|^2-\frac{t-2t^3}{(1-t^2)^{3/2}}\langle x_i^*-x_j^*,v_i-v_j\rangle+\|v_i-v_j\|^2
}{
(1-t^2) \|x_i^*-x_j^*\|^2+2t\sqrt{1-t^2}\langle x_i^*-x_j^*,v_i-v_j\rangle+t^2\|v_i-v_j\|^2
}.
$$
For any $0<\alpha\leq t$, we use the upper bounds $\|x_i^*-x_j^*\|\leq 2$, $\|v_i-v_j\|\leq 2$,  $|\langle x_i^*-x_j^*,v_i-v_j\rangle|=|\langle x_i^*,v_j\rangle+\langle x_j^*,v_i\rangle|\leq 2$ and $\alpha-2\alpha^3\leq \alpha(1-\alpha^2)$, as well as the lower bound $\|x_i^*-x_j^*\|\geq 2/\sqrt N$ provided by Proposition \ref{le:sep}, to obtain
$$
\ddot g(\alpha) \leq N(N-1)\frac{ 1}{\frac4N(1-t^2)-4t}\left(\frac8{1-t^2}+\frac{2t}{\sqrt{1-t^2}} \right).
$$
Finally, we use the inequality $1-t^2-Nt\geq (1-t^2)/4$ for any $0<t\leq 1/(2N)$ to obtain   
\eq
\label{ddotgineqII}
\ddot g(\alpha) \leq N^3\left( \frac8{(1-t^2)^2}+ \frac{2t}{(1-t^2)^{3/2}} \right)\leq 10N^3.
\qe
Together with \eqref{Taylor} this yields \eqref{toprove}, and the proof of  Proposition~\ref{th:bh} is therefore complete.


\end{proof}

\bibliographystyle{plainnat}


\end{document}